\documentclass[11pt]{article}
\usepackage{authblk}
\usepackage{amsmath}
\usepackage{amssymb}
\usepackage{amsthm}
\usepackage{hyperref}
\usepackage{fullpage}
\usepackage{graphics}
\usepackage{bbm}
\usepackage{dsfont}
\usepackage{enumitem}
\usepackage{caption, subcaption}
\usepackage{tikz,color}
\usetikzlibrary{shapes.misc}
\usepackage{multirow}
\usepackage{mathrsfs}

 \newtheorem{theorem}{Theorem}[section]
 \newtheorem{corollary}[theorem]{Corollary}

 \newtheorem{lemma}[theorem]{Lemma}

 \newtheorem{proposition}[theorem]{Proposition}
 
 \theoremstyle{definition}
 \newtheorem{example}[theorem]{Example}

 \newtheorem{definition}[theorem]{Definition}
 
 \newtheorem{remark}[theorem]{Remark}

\newcommand{\PF}[1]{\mathrm{PF}_{#1}}
\newcommand{\FPF}[1]{\mathrm{FPF}\left( #1 \right)}
 
\newcommand{\OPF}[2]{\mathcal{O}_{\PF{#1}}\left( #2 \right)}
\newcommand{\OFPF}[2]{\mathcal{O}_{\FPF{#1}}\left( #2 \right)}

\newcommand{\OOPF}[1]{\mathcal{O}_{\PF{#1}}}
\newcommand{\OOFPF}[1]{\mathcal{O}_{\FPF{#1}}}

\newcommand{\FibFPF}[2]{\mathcal{O}^{-1}_{\FPF{#1}}\left( #2 \right)}
\newcommand{\FibPF}[2]{\mathcal{O}^{-1}_{\PF{#1}}\left( #2 \right)}

\newcommand{\Ham}[1]{\mathcal{H}\left( #1 \right)}

\newcommand{\inccyc}[2]{#1(#1+1) \cdots #2 12\cdots (#1-1)}
\newcommand{\IncCyc}[2]{\mathrm{IncCyc}^{#1, #2}}
\newcommand{\deccyc}[2]{#1(#1-1) \cdots 1 #2(#2-1) \cdots (#1+1)}
\newcommand{\DecCyc}[2]{\mathrm{DecCyc}^{#1, #2}}

\newcommand{\CycPF}[1]{\mathrm{CycPF}_{#1}}

\newcommand{\comp}[1]{\mathrm{comp}\left( #1 \right)}
\newcommand{\Comp}[1]{\mathrm{Comp}_{#1}}
\newcommand{\Perm}[1]{\mathrm{Perm}\left( #1 \right)}

\newcommand{\inv}[2]{\mathrm{inv}_{#1}\left( #2 \right)}
\newcommand{\InvSeq}[1]{\mathrm{InvSeq\left( #1 \right)}}

\newcommand{\disp}[1]{\mathrm{disp}\left( #1 \right)}

\renewcommand{\thanks}[1]{} 

\title{On friendship and cyclic parking functions}
\author[1]{Yujia Kang\thanks{Yujia.Kang20@student.xjtlu.edu.cn}}
\author[2]{Thomas Selig\thanks{Thomas.Selig@xjtlu.edu.cn}}
\author[1]{Guanyi Yang\thanks{Guanyi.Yang20@student.xjtlu.edu.cn}}
\author[3]{Yanting Zhang\thanks{Yanting.Zhang@maths.ox.ac.uk}}
\author[2]{Haoyue Zhu\thanks{Haoyue.Zhu18@student.xjtlu.edu.cn}}
\affil[1]{School of Mathematics and Physics, Xi'an Jiaotong-Livepool University} 
\affil[2]{School of Advanced Technology, Xi'an Jiaotong-Livepool University} 
\affil[3]{Mathematical Institute, University of Oxford}
\date{\today}

\begin{document}

\maketitle

\begin{abstract}
In parking problems, a given number of cars enter a one-way street sequentially, and try to park according to a specified preferred spot in the street. Various models are possible depending on the chosen rule for \emph{collisions}, when two cars have the same preferred spot. In classical parking functions, if a car's preferred spot is already occupied by a previous car, it drives forward and looks for the first unoccupied spot to park.

In this work, we introduce a variant of classical parking functions, called \emph{friendship parking functions}, which imposes additional restrictions on where cars can park. Namely, a car can only end up parking next to cars which are its friends (friendship will correspond to adjacency in an underlying graph). We characterise and enumerate such friendship parking functions according to their outcome permutation, which describes the final configuration when all cars have parked. We apply this to the case where the underlying friendship graph is the cycle graph. Finally, we consider a subset of classical parking functions, called \emph{cyclic parking functions}, where cars end up in an increasing cyclic order. We enumerate these cyclic parking functions and exhibit a bijection to permutation components.
\end{abstract}



\section{Introduction}\label{sec:intro}
In this section, we introduce parking functions and friendship parking functions. Throughout the paper, $n$ represents a positive integer, and we denote $[n] := \{1, \cdots, n\}$. Permutations will play a key role in our paper. We denote $S_n$ the set of permutations of the set $[n]$. For $\pi \in S_n$, we write $\pi = \pi_1 \cdots \pi_n$ using standard one-line notation. We also denote $\pi^{-1}$ the inverse permutation of $\pi$. That is, for any $i \in [n]$, $\pi^{-1}_i$ is the unique $j \in [n]$ such that $\pi_j = i$.

\subsection{Parking functions and their variations}\label{subsec:PF_and_variations}


A \emph{parking preference} is a vector $p = (p_1, \cdots, p_n) \in [n]^n$. We will regard $p_i$ as the preferred spot of car $i$ in a car park with spots labelled $1, \cdots, n$. Given a parking preference $p$, the classical parking process for $p$ is described as follows.

Cars enter the car park sequentially in order $1, \cdots, n$, and initially all spots in the car park are unoccupied. When car $i$ enters, it will occupy (park in) the first spot $k \geq p_i$ which is unoccupied at that point. If no such spot exists, car $i$ exits the car park, having failed to park. We say that $p$ is a \emph{classical parking function} if all $n$ cars are able to park through this process. We denote $\PF{n}$ the set of all classical parking functions with $n$ cars/spots.

The \emph{outcome map} of classical parking functions is the map $\OOPF{n} : \PF{n} \rightarrow S_n$ which associates a permutation $\OPF{n}{p} = \pi = \pi_1 \cdots \pi_n$ of $[n]$ to a classical parking function $p$ as follows. For any $i \in [n]$, we set $\pi_i$ to be the label of the car ending up in spot $i$ after all cars have parked in the classical parking process for $p$.

The \emph{displacement vector} of a classical parking function $p \in \PF{n}$ is defined by $\disp{p} = (d_1, d_2, \cdots, d_n)$, where $d_i := \pi^{-1}_i - p_i$ for all $i \in [n]$. Note that for $i \in [n]$, $d_i$ measures the distance that car $i$ has travelled between its preferred spot $p_i$ and the spot $\pi^{-1}_i$ where it finally parks. We say that $d_i$ is the displacement of car $i$ in the parking function $p$. The total displacement of $p$ is defined by $\vert \disp{p} \vert := \sum_{i=1}^{n} d_i$.

\begin{example}\label{ex:classical_PF}
Consider the parking preference $p = (3,1,1,2)$, and let us illustrate the classical parking process for $p$.
\begin{itemize}
  \item Car $1$ has preference $p_1 = 3$. Spot $3$ is unoccupied at that point, so car $1$ parks in spot $3$ (in fact, car $1$ will always park in its preferred spot), with displacement $d_1 = 0$.
  \item Car $2$ has preference $p_2 = 1$. Spot $1$ is unoccupied at that point, so car $2$ parks in spot $1$, with displacement $d_2 = 0$.
  \item Car $3$ has preference $p_3 = 1$. However, spot $1$ is occupied at this point (by car $2$), so car $3$ cannot park there. Instead, it continues to drive forward and finds the first unoccupied spot after its preference, which is spot $2$, and so car $3$ parks there. The displacement is given by $d_3 = 2 - 1 = 1$.
  \item Lastly, car $4$ has preference $p_4 = 2$. Again, spot $2$ is occupied (by car $3$), as is spot $3$ (by car $1$). Thus, car $4$ will park in spot $4$, which is in fact the only unoccupied spot at this point. The displacement is given by $d_4 = 4 - 2 = 2$.
\end{itemize}
Finally, the cars are all able to park, so $p = (3,1,1,2)$ is a classical parking function. We also get the outcome $\OPF{n}{p} = 2314$, and the displacement vector $\disp{p} = (0,0,1,2)$. The entire classical parking process is illustrated in Figure~\ref{fig:ex_PF_valid} below.

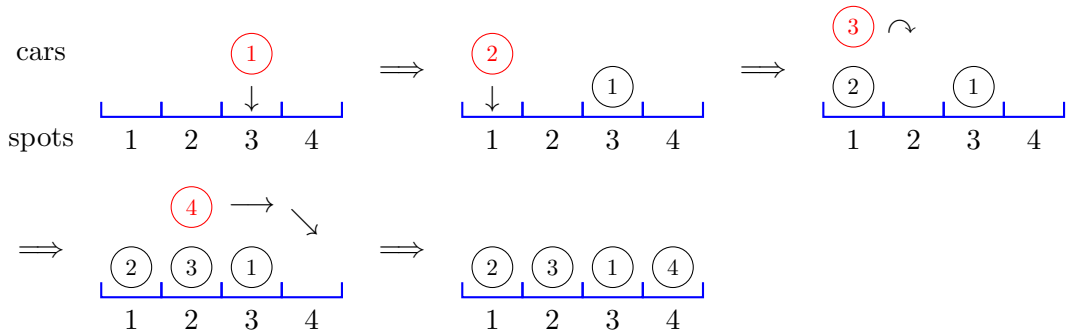
\begin{figure}[ht]
 \centering
 
  \begin{tikzpicture}[scale=0.2]
    
    \node at (-3,1.2) {cars};
    \node at (-3,-4.5) {spots};       
    
    \foreach \x in {1,...,4}
	  \node at (-1+4*\x,-4.5) {$\x$};
    \draw [thick, color=blue] (1,-2)--(1,-3)--(5,-3)--(5,-2);
    \draw [thick, color=blue] (5,-2)--(5,-3)--(9,-3)--(9,-2);
    \draw [thick, color=blue] (9,-2)--(9,-3)--(13,-3)--(13,-2);
    \draw [thick, color=blue] (13,-2)--(13,-3)--(17,-3)--(17,-2);
    \node [draw, circle, color=red, scale=0.8pt] (1) at (11,1.2) {$1$};
    \node at (11, -1.8) {$\downarrow$};
    \node at (21, 0) {$\Longrightarrow$};

    \begin{scope}[shift={(24,0)}]
    \foreach \x in {1,...,4}
	  \node at (-1+4*\x,-4.5) {$\x$};
    \draw [thick, color=blue] (1,-2)--(1,-3)--(5,-3)--(5,-2);
    \draw [thick, color=blue] (5,-2)--(5,-3)--(9,-3)--(9,-2);
    \draw [thick, color=blue] (9,-2)--(9,-3)--(13,-3)--(13,-2);
    \draw [thick, color=blue] (13,-2)--(13,-3)--(17,-3)--(17,-2);
    \node [draw, circle, scale=0.8pt] (1) at (11,-1) {$1$};
    \node [draw, circle, color=red, scale=0.8pt] (2) at (3,1.2) {$2$};
    \node at (3, -1.8) {$\downarrow$};
    \node at (21, 0) {$\Longrightarrow$};
    \end{scope}

    \begin{scope}[shift={(48,0)}]
    \foreach \x in {1,...,4}
	  \node at (-1+4*\x,-4.5) {$\x$};
    \draw [thick, color=blue] (1,-2)--(1,-3)--(5,-3)--(5,-2);
    \draw [thick, color=blue] (5,-2)--(5,-3)--(9,-3)--(9,-2);
    \draw [thick, color=blue] (9,-2)--(9,-3)--(13,-3)--(13,-2);
    \draw [thick, color=blue] (13,-2)--(13,-3)--(17,-3)--(17,-2);
    \node [draw, circle, scale=0.8pt] (1) at (11,-1) {$1$};
    \node [draw, circle, scale=0.8pt] (2) at (3,-1) {$2$};
    \node [draw, circle, color=red, scale=0.8pt] (3) at (3,3) {$3$};
    \node at (6.2, 3) {$\curvearrowright$};
    \end{scope}

    \begin{scope}[shift={(0,-12)}]
    \node at (-3, 0) {$\Longrightarrow$};
    \foreach \x in {1,...,4}
	  \node at (-1+4*\x,-4.5) {$\x$};
    \draw [thick, color=blue] (1,-2)--(1,-3)--(5,-3)--(5,-2);
    \draw [thick, color=blue] (5,-2)--(5,-3)--(9,-3)--(9,-2);
    \draw [thick, color=blue] (9,-2)--(9,-3)--(13,-3)--(13,-2);
    \draw [thick, color=blue] (13,-2)--(13,-3)--(17,-3)--(17,-2);
    \node [draw, circle, scale=0.8pt] (1) at (11,-1) {$1$};
    \node [draw, circle, scale=0.8pt] (2) at (3,-1) {$2$};
    \node [draw, circle, scale=0.8pt] (3) at (7,-1) {$3$};
    \node [draw, circle, color=red, scale=0.8pt] (4) at (7,3) {$4$};
    \node at (11, 3) {$\longrightarrow$};
    \node at (14.5, 2) {$\searrow$};
    \node at (21, 0) {$\Longrightarrow$};
    \end{scope}

    \begin{scope}[shift={(24,-12)}]
    \foreach \x in {1,...,4}
	  \node at (-1+4*\x,-4.5) {$\x$};
    \draw [thick, color=blue] (1,-2)--(1,-3)--(5,-3)--(5,-2);
    \draw [thick, color=blue] (5,-2)--(5,-3)--(9,-3)--(9,-2);
    \draw [thick, color=blue] (9,-2)--(9,-3)--(13,-3)--(13,-2);
    \draw [thick, color=blue] (13,-2)--(13,-3)--(17,-3)--(17,-2);
    \node [draw, circle, scale=0.8pt] (1) at (11,-1) {$1$};
    \node [draw, circle, scale=0.8pt] (2) at (3,-1) {$2$};
    \node [draw, circle, scale=0.8pt] (3) at (7,-1) {$3$};
    \node [draw, circle, scale=0.8pt] (4) at (15,-1) {$4$};
    \end{scope}
    
  \end{tikzpicture}
    \caption{Illustrating the parking process for the classical parking function $p = (3,1,1,2)$.} \label{fig:ex_PF_valid}
\end{figure}
\end{example}

Parking functions were originally introduced by Konheim and Weiss~\cite{KonWeiss} in their study of \emph{hashing functions}. Since then, they have been a popular research topic in Mathematics and Computer Science, with rich connections to a variety of fields such as graph theory, representation theory, hyperplane arrangements, discrete geometry, and the Abelian sandpile model~\cite{ChevPyl, CR, StanHyp, PitStan}. We refer the interested reader to the excellent survey by Yan~\cite{YanSurvey}.

We may notice that in classical parking functions, the handling of \emph{collisions} -- when a car's preferred spot is already occupied when it enters the car park -- appears quite flexible. Indeed, many variations of parking functions with different parking rules have been discussed in the literature.  We list some of them briefly here. For a more detailed survey of generalisations and open problems related to parking functions, see~\cite{HarrisSurvey}.

\begin{itemize}
\item \textbf{Defective parking functions}~\cite{CamDefPF}. This model follows classical parking rules, with $m$ cars and $n$ spots, and $k$ cars being unable to park ($k$ is called the \emph{defect} of the parking function).
\item \textbf{Naples parking functions}~\cite{HarrisNaples1, HarrisNaples2, TianNaples}. This model allows for some backward movement of cars from their preferred spot if it is already occupied.
\item \textbf{Parking assortments and parking sequences} ~\cite{AdeYan, HarrisSeq2, EhrHapp, HarrisSeq1}. These models involve cars of different sizes, with the car park having just enough space for all cars to park.
\item \textbf{Subset parking functions}~\cite{Spiro}. In this model, cars have a preferred ordered subset of spots in which they can park. See also~\cite{CMMY} for an in-depth treatment of \textbf{interval parking functions}, where each preference subset is a sub-interval of $[n]$.
\item \textbf{Vector parking functions}, or $\mathbf{u}$-parking functions~\cite{Yan2023, Yin2023}. For a vector $u = (u_1, \cdots, u_n)$, we say that a parking preference $p = (p_1, p_2,\ldots, p_n)$ is a $u$-parking function if the non-decreasing rearrangement $a = (a_1, \cdots, a_n)$ of $p$ satisfies $a_i \leq u_i$ for all $i \in [n]$. These correspond to classical parking functions when $u = (1, 2, \ldots, n)$ (see e.g.~\cite[Section~1.1]{YanSurvey} for this characterisation of classical parking functions).
\item \textbf{Graphical parking functions} on some graph $G$, or $G$-parking functions~\cite{PostPF}. In this model, cars park on the vertices of the graph $G$ instead of in a one-way street.
\item \textbf{Higher-dimensional parking functions}. There are several relevant models in this area, such as the \textbf{$(p,q)$-parking functions} of Cori and Poulalhon~\cite{CoriPou}, or the \textbf{tiered parking functions} of Dukes~\cite{DukesSplit}. Both of these models involve cars of different colours or tiers, with extra conditions on where a car can park depending on its tier/colour. Snider and Yan~\cite{YanUpq, Yan2023} then defined a concept of \textbf{multidimensional $U$-parking functions}, where $U$ is a collection of multidimensional vectors. These generalise the $(p, q)$-parking functions of Cori and Poulalhon and the (one-dimensional) vector parking functions, and also connect to the graphical parking functions discussed above.
\item \textbf{MVP parking functions}~\cite{HarrisMVP, SeligZhuMVP}. In this model, priority is given to cars which enter later in the parking process. That is, when a car enters the car park, if its preferred spot is occupied by a previous car, it will ``bump'' the earlier car out of that spot, forcing it to drive on and find a new spot to park in.
\end{itemize}

In this paper, we consider a new variation on parking functions, which we call \emph{friendship parking functions}. Informally, cars follow the classical parking process, but can only park next to cars which they are ``friends'' with (friendship will be determined by adjacency in an underlying graph).

\subsection{Friendship parking functions}\label{subsec:FPF_intro}

In this section we formally define the concept of friendship parking functions. Let $G$ be a graph with vertex set $[n]$. We say that two vertices $i$ and $j$ are \emph{adjacent} in the graph $G$ if $\{i,j\}$ is an edge of $G$. As in the classical case, we consider a parking preference $p = (p_1, \cdots, p_n) \in [n]^n$. The \emph{friendship parking process} is described as follows. 

Cars enter sequentially in order $1, \cdots, n$ into a car park with spots labelled $1, \cdots, n$. All spots are initially unoccupied. For $i \geq 1$, we say that a spot $k$ is \emph{available} for car $i$ if it is unoccupied when car $i$ enters the car park, and each of the two neighbouring spots $k-1$ and $k+1$ are either unoccupied, or occupied by a car $j$ (necessarily $j<i$) which is adjacent to $i$ in the graph $G$. Here we consider spots $0$ and $n+1$ to be unoccupied by convention when handling the cases $k = 1$ or $k=n$. When car $i$ enters the car park, it occupies (parks in) the first spot $k \geq p_i$ which is available for $i$. If no such spot exists, the car exits the car park, having failed to park.

We say that $p$ is a \emph{friendship parking function} for $G$, or $G$-friendship parking function, if all $n$ cars are able to park following this process. We denote $\FPF{G}$ the set of all $G$-friendship parking functions. As in the classical case, we can define the outcome map $\OOFPF{G}: \FPF{G} \rightarrow S_n$. For a $G$-friendship parking function $p$, we let $\pi = \pi_1 \cdots \pi_n = \OPF{G}{p}$ be the permutation such that for any $i \in [n]$, $\pi_i$ is the label of the car ending up in spot $i$ when running the friendship parking process for $p$.

Note that if $G$ is the complete graph on $n$ vertices (with edges between any pair of distinct vertices), the notions of classical and friendship parking functions are identical. Indeed, in this case, a spot is available for $i$ if, and only if, it is unoccupied when $i$ enters the car park.

\begin{example}\label{ex:FPF_valid}
Let $G$ be the cycle on $n=4$ vertices, with edges $\{1,2\}, \{2,3\}, \{3,4\}, \{4,1\}$. Consider the parking preference $p = (2,1,2,2)$, and let us illustrate the $G$-friendship parking process for $p$. 

\begin{itemize}
  \item Car $1$ has preference $p_1 = 2$. Spot $2$ is available for car $1$ since both adjacent spots are both unoccupied (this will always be the case for car $1$'s preferred spot). Therefore car $1$ parks in spot $2$.
  \item Car $2$ has preference $p_2 = 1$. The neighbouring spot $2$ is occupied by car $1$ at this point, but vertices $1$ and $2$ are adjacent in the graph $G$, so spot $1$ is available for car $2$, allowing it to park there.
  \item Car $3$ has preference $p_3 = 2$. However, spot $2$ is occupied at this point (by car $1$), so is not available for $3$. Spot $3$ is unoccupied, but the adjacent spot $2$ is occupied by car $1$, and $1$ and $3$ are not adjacent in $G$. Therefore spot $3$ is not available for car $3$. Finally, spot $4$ is also unoccupied, as is its neighbouring spot $3$, which implies that it is available, and thus car $3$ parks in spot $4$.
  \item Lastly, car $4$ has preference $p_4 = 2$. Spot $2$ is occupied, so car $4$ can't park there. But spot $3$ is unoccupied, and the cars occupying its neighbouring spots $2$ (car $1$) and $4$ (car $3$) are both adjacent to $4$ in $G$, which means spot $3$ is available for car $4$, so it parks there.
\end{itemize}
Finally, the cars are all able to park, so $p = (2,1,2,2)$ is a $G$-friendship parking function, and we get the outcome $\OFPF{G}{p} = 2143$. The entire friendship parking process is illustrated in Figure~\ref{fig:ex_FPF_valid} below.

\begin{figure}[ht]
 \centering 
  \begin{tikzpicture}[scale=0.2]
    
    \node at (-3,1.2) {cars};
    \node at (-3,-4.5) {spots};       
    
    \foreach \x in {1,...,4}
	  \node at (-1+4*\x,-4.5) {$\x$};
    \draw [thick, color=blue] (1,-2)--(1,-3)--(5,-3)--(5,-2);
    \draw [thick, color=blue] (5,-2)--(5,-3)--(9,-3)--(9,-2);
    \draw [thick, color=blue] (9,-2)--(9,-3)--(13,-3)--(13,-2);
    \draw [thick, color=blue] (13,-2)--(13,-3)--(17,-3)--(17,-2);
    \node [draw, circle, color=red, scale=0.8pt] (1) at (7,1.2) {$1$};
    \draw [->] (7,-0.9)--(7,-2.6);
    \node at (21, 0) {$\Longrightarrow$};

    \begin{scope}[shift={(24,0)}]
    \foreach \x in {1,...,4}
	  \node at (-1+4*\x,-4.5) {$\x$};
    \draw [thick, color=blue] (1,-2)--(1,-3)--(5,-3)--(5,-2);
    \draw [thick, color=blue] (5,-2)--(5,-3)--(9,-3)--(9,-2);
    \draw [thick, color=blue] (9,-2)--(9,-3)--(13,-3)--(13,-2);
    \draw [thick, color=blue] (13,-2)--(13,-3)--(17,-3)--(17,-2);
    \node [draw, circle, scale=0.8pt] (1) at (7,-1) {$1$};
    \node [draw, circle, color=red, scale=0.8pt] (1) at (3,1.2) {$2$};
    \draw [->] (3,-0.9)--(3,-2.6);
    \node at (21, 0) {$\Longrightarrow$};
    \end{scope}

    \begin{scope}[shift={(48,0)}]
    \foreach \x in {1,...,4}
	  \node at (-1+4*\x,-4.5) {$\x$};
    \draw [thick, color=blue] (1,-2)--(1,-3)--(5,-3)--(5,-2);
    \draw [thick, color=blue] (5,-2)--(5,-3)--(9,-3)--(9,-2);
    \draw [thick, color=blue] (9,-2)--(9,-3)--(13,-3)--(13,-2);
    \draw [thick, color=blue] (13,-2)--(13,-3)--(17,-3)--(17,-2);
    \node [draw, circle, scale=0.8pt] (1) at (7,-1) {$1$};
    \node [draw, circle, scale=0.8pt] (2) at (3,-1) {$2$};
    \node [draw, circle, color=red, scale=0.8pt] (3) at (7,3) {$3$};
    \node at (11,-1) {{\LARGE$\times$}};
    \draw [->, out=0, in=90] (9,3.6) to (10.5,1.2);
    \draw [->, out=30, in=150] (11.5,1.2) to (15,1.2);
    \draw [->] (15,0.4)--(15,-2.6);
    \end{scope}

    \begin{scope}[shift={(0,-12)}]
    \node at (-3, 0) {$\Longrightarrow$};
    \foreach \x in {1,...,4}
	  \node at (-1+4*\x,-4.5) {$\x$};
    \draw [thick, color=blue] (1,-2)--(1,-3)--(5,-3)--(5,-2);
    \draw [thick, color=blue] (5,-2)--(5,-3)--(9,-3)--(9,-2);
    \draw [thick, color=blue] (9,-2)--(9,-3)--(13,-3)--(13,-2);
    \draw [thick, color=blue] (13,-2)--(13,-3)--(17,-3)--(17,-2);
    \node [draw, circle, scale=0.8pt] (1) at (7,-1) {$1$};
    \node [draw, circle, scale=0.8pt] (2) at (3,-1) {$2$};
    \node [draw, circle, scale=0.8pt] (3) at (15,-1) {$3$};
    \node [draw, circle, color=red, scale=0.8pt] (4) at (7,3) {$4$};
    \draw [->, out=-30, in=90] (9.2,3) to (11,-2);
    \node at (21, 0) {$\Longrightarrow$};
    \end{scope}

    \begin{scope}[shift={(24,-12)}]
    \foreach \x in {1,...,4}
	  \node at (-1+4*\x,-4.5) {$\x$};
    \draw [thick, color=blue] (1,-2)--(1,-3)--(5,-3)--(5,-2);
    \draw [thick, color=blue] (5,-2)--(5,-3)--(9,-3)--(9,-2);
    \draw [thick, color=blue] (9,-2)--(9,-3)--(13,-3)--(13,-2);
    \draw [thick, color=blue] (13,-2)--(13,-3)--(17,-3)--(17,-2);
    \node [draw, circle, scale=0.8pt] (1) at (7,-1) {$1$};
    \node [draw, circle, scale=0.8pt] (2) at (3,-1) {$2$};
    \node [draw, circle, scale=0.8pt] (3) at (15,-1) {$3$};
    \node [draw, circle, scale=0.8pt] (4) at (11,-1) {$4$};
    \end{scope}
  \end{tikzpicture}
  \caption{Illustrating the parking process for the friendship parking function $p = (2,1,2,2)$. \label{fig:ex_FPF_valid}}
\end{figure}
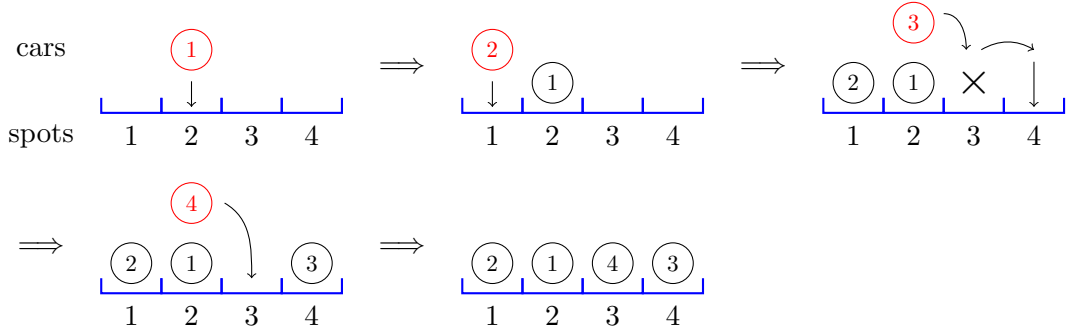
\end{example}

\begin{example}\label{ex:FPF_invalid}
Let $G$ again be the cycle graph on $4$ vertices, as in Example~\ref{ex:FPF_valid}, and consider the parking preference $p = (4,2,2,1)$. Car $1$ parks in spot $4$, and car $2$ parks in spot $2$ (both neighbouring spots are unoccupied). However, car $3$ cannot park. Indeed, the preference of car $3$ is $p_3 = 2$. At this point, the only unoccupied spot $k \geq 2$ is $k=3$, but the neighbouring spot $4$ is occupied by car $1$, which is not adjacent to $3$ in $G$. Therefore spot $3$ is not available for car $3$, and so there are no available spots on or after car $3$'s preference for it to park in. Thus $p$ is \emph{not} a $G$-friendship parking function. This parking process is illustrated on Figure~\ref{fig:ex_FPF_invalid} below.

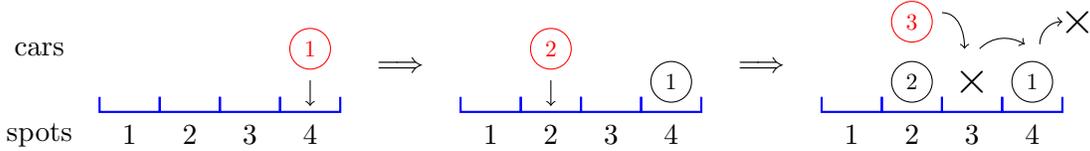
\begin{figure}[ht]
 \centering 
  \begin{tikzpicture}[scale=0.2]
    
    \node at (-3,1.2) {cars};
    \node at (-3,-4.5) {spots};       
    
    \foreach \x in {1,...,4}
	  \node at (-1+4*\x,-4.5) {$\x$};
    \draw [thick, color=blue] (1,-2)--(1,-3)--(5,-3)--(5,-2);
    \draw [thick, color=blue] (5,-2)--(5,-3)--(9,-3)--(9,-2);
    \draw [thick, color=blue] (9,-2)--(9,-3)--(13,-3)--(13,-2);
    \draw [thick, color=blue] (13,-2)--(13,-3)--(17,-3)--(17,-2);
    \node [draw, circle, color=red, scale=0.8pt] (1) at (15,1.2) {$1$};
    \draw [->] (15,-0.9)--(15,-2.6);
    \node at (21, 0) {$\Longrightarrow$};

    \begin{scope}[shift={(24,0)}]
    \foreach \x in {1,...,4}
	  \node at (-1+4*\x,-4.5) {$\x$};
    \draw [thick, color=blue] (1,-2)--(1,-3)--(5,-3)--(5,-2);
    \draw [thick, color=blue] (5,-2)--(5,-3)--(9,-3)--(9,-2);
    \draw [thick, color=blue] (9,-2)--(9,-3)--(13,-3)--(13,-2);
    \draw [thick, color=blue] (13,-2)--(13,-3)--(17,-3)--(17,-2);
    \node [draw, circle, scale=0.8pt] (1) at (15,-1) {$1$};
    \node [draw, circle, color=red, scale=0.8pt] (1) at (7,1.2) {$2$};
    \draw [->] (7,-0.9)--(7,-2.6);
    \node at (21, 0) {$\Longrightarrow$};
    \end{scope}

    \begin{scope}[shift={(48,0)}]
    \foreach \x in {1,...,4}
	  \node at (-1+4*\x,-4.5) {$\x$};
    \draw [thick, color=blue] (1,-2)--(1,-3)--(5,-3)--(5,-2);
    \draw [thick, color=blue] (5,-2)--(5,-3)--(9,-3)--(9,-2);
    \draw [thick, color=blue] (9,-2)--(9,-3)--(13,-3)--(13,-2);
    \draw [thick, color=blue] (13,-2)--(13,-3)--(17,-3)--(17,-2);
    \node [draw, circle, scale=0.8pt] (1) at (15,-1) {$1$};
    \node [draw, circle, scale=0.8pt] (2) at (7,-1) {$2$};
    \node [draw, circle, color=red, scale=0.8pt] (3) at (7,3) {$3$};
    \node at (11,-1) {{\LARGE$\times$}};
    \draw [->, out=0, in=90] (9,3.6) to (10.5,1.2);
    \draw [->, out=45, in=150] (11.5,1.2) to (14.5,1.5);
    \draw [->, out=90, in=180] (15.5,1.5) to (17,3);
    \node at (18, 3) {{\LARGE$\times$}};
    \end{scope}
  \end{tikzpicture}
  \caption{Illustrating the friendship parking process for the parking preference $p = (4,2,2,1)$. Here, car $3$ is unable to park, so $p$ is not a friendship parking function. \label{fig:ex_FPF_invalid}}
\end{figure}
\end{example}

It is worth noting that in the previous example, $p = (4,2,2,1)$ is a classical parking function (we get the outcome $\OPF{4}{p} = 4231$). This shows that our new notion of friendship parking functions is a meaningful variant of classical parking functions.

\subsection{Contents of the paper}\label{subsec:contents}

Our paper is organised as follows. In Section~\ref{sec:FPF} we study this new friendship parking process. We begin in Section~\ref{subsec:Ham_FPF_PF} by showing a link between $G$-friendship parking functions and \emph{Hamiltonian paths} of the graph $G$ (Proposition~\ref{pro:FPF_HamPath}), and by giving some results connecting friendship and classical parking functions (Propositions~\ref{pro:FPF_implies_PF} and \ref{pro:PF+outcome_implies_FPF}). Section~\ref{subsec:general_counting} then gives a characterisation and enumeration formula for friendship parking functions on general graphs (Theorem~\ref{thm:charact_FPF_outcome} and Corollary~\ref{cor:num_GFPF}). In Section~\ref{subsec:cycle} we apply these results to the case where $G$ is the cycle graph on $n$ vertices, yielding the enumeration formula from Theorem~\ref{thm:count_cn_PFP}.

In Section~\ref{sec:cyclicPF} we return to the study of classical parking functions, considering those whose outcome is an increasing cycle of the form $\IncCyc{i}{n} := \inccyc{i}{n}$ for some $i \in [n]$. We call such parking functions \emph{cyclic parking functions}, and count them (Proposition~\ref{pro:outcome_enum_CycPF} and Corollary~\ref{cor:total_number_CycPF}). We then describe a bijection between cyclic parking functions and \emph{permutation components}, which maps the displacement of cars to inversions in the corresponding permutation (Theorem~\ref{thm:bij_cycPF_perm_comp}). Finally, Section~\ref{sec:future} gives a brief summary of our results, and outlines some possible future research directions on friendship parking functions.


\section{Characterising and counting friendship parking functions}\label{sec:FPF}

In this section, we give characterisations and enumeration formulas for friendship parking functions on general graphs $G$. We then apply these to the case where $G$ is the cycle graph on $n$ vertices. 

\subsection{Friendship parking functions and Hamiltonian paths}\label{subsec:Ham_FPF_PF}

We begin by exhibiting a connection between $G$-friendship parking functions and \emph{Hamiltonian paths} of the graph $G$. To do this, we first recall some relevant concepts. Given a graph $G$, with vertex set $[n]$, a \emph{path} of $G$ is a sequence of distinct vertices $\mathcal{P} = (v_1, v_2,\cdots, v_k) \in [n]^k$ such that $\{v_i,v_{i+1}\}$ is an edge of $E$ for each $i \in \{1, \cdots, k-1\}$. Moreover, $\mathcal{P}$ is called a \emph{Hamiltonian} path if $k = n$, i.e.\ every vertex of $G$ appears in $\mathcal{P}$.

Recall that for a $G$-friendship parking function $p \in \FPF{G}$, we can define its outcome $\pi = \OFPF{G}{p} = \pi_1 \cdots \pi_n$, which is a permutation of all the cars (for $i \in [n]$, car $\pi_i$ is the car that ends up in spot $i$). In friendship parking functions, the relationship among cars is indicated by a graph $G$ and the permutation reflects that cars (vertices) that can park next to each other must be ``friends'' (adjacent in the graph $G$). We can therefore obtain the following fact. 

\begin{proposition}\label{pro:FPF_HamPath}
Let $G=(V,E)$ be a graph. Then the set of $G$-friendship parking functions is non-empty if, and only if, $G$ contains a Hamiltonian path.
\end{proposition}

\begin{proof}
Suppose there exists some $p \in \FPF{G}$, and consider its outcome $\OFPF{G}{p} = \pi_1 \cdots \pi_n$. We claim that $\mathcal{P} := (\pi_1, \cdots, \pi_n)$ is a Hamiltonian path of $G$. Indeed, consider a car $i = \pi_k$ which parks in some spot $k \leq n-1$ in the friendship parking process. We wish to show that $j := \pi_{k+1}$ is adjacent to $i$ in $G$. If $j<i$, then car $j$ has already parked in spot $k+1$ when car $i$ arrives. This means that for $i$ to park in spot $k$, vertices $i$ and $j$ must be adjacent in $G$ (since spot $k$ must be available for car $i$ when it arrives). The same applies if $j>i$: when car $j$ arrives, car $i$ has already parked in spot $k$, so $j$ can only park in spot $k+1$ if $i$ and $j$ are adjacent in $G$. Therefore $\mathcal{P}$ is a Hamiltonian path, as desired.

Conversely, suppose that $G$ contains a Hamiltonian path, say $\mathcal{P}= (v_1, v_2,\cdots, v_n) \in [n]^n$ where the $v_i$ are distinct. By construction, $\pi := v_1 \cdots v_n$ is a permutation of $[n]$. Define a parking preference $p = (p_1, \cdots, p_n)$ by $p_i = \pi^{-1}_i$ for all $i \in [n]$. Then it is straightforward to check that $p \in \FPF{G}$ and $\OFPF{G}{p} = \pi$. Indeed, in the friendship parking process for $p$, each car can simply park in its preferred spot (since $\mathcal{P}$ is a Hamiltonian path, these are always available).
\end{proof}

Proposition~\ref{pro:FPF_HamPath} establishes a connection between the outcomes of $G$-friendship parking functions and Hamiltonian paths of $G$. With a slight abuse of notation, we will henceforth consider a Hamiltonian path of $G$ to be a permutation $\pi = \pi_1 \cdots \pi_n$ of $[n]$ such that $\{\pi_i, \pi_{i+1}\}$ is an edge of $G$ for all $i \in [n-1]$. We denote $\Ham{G}$ the set of Hamiltonian paths of $G$. We end this section by stating two fairly straightforward propositions that connect classical and friendship parking functions.

\begin{proposition}\label{pro:FPF_implies_PF}
We have $\FPF{G} \subseteq \PF{n}$.
\end{proposition}

\begin{proof}
We first recall a well-known characterisation of classical parking functions. For a parking preference $p = (p_1, \cdots, p_n)$, we denote $\tilde{p} = (\tilde{p}_1, \cdots, \tilde{p}_n)$ the \emph{non-decreasing rearrangement} of $p$. Then $p \in \PF{n}$ if, and only if, we have $\tilde{p}_i \leq i$ for all $i \in [n]$ (see e.g.~\cite[Section~1.1]{YanSurvey}). 

Now let $p \in \FPF{G}$ be a $G$-friendship parking function, and denote $\tilde{p}$ its non-decreasing rearrangement. Seeking contradiction, suppose there is some $i \in [n]$ such that $\tilde{p}_i > i$. By construction, this means that $\left\vert \{j \in [n];\, p_j \leq i \} \right\vert \leq i-1$. But in the friendship (or indeed classical) parking process, each car necessarily ends up parking on or after its preferred spot. In particular, for any $i \in [n]$, there must be at least $i$ cars whose preference is one of the first $i$ spots (since exactly $i$ cars must end up in these spots). This yields the desired contradiction with the above inequality.
\end{proof}

\begin{proposition}\label{pro:PF+outcome_implies_FPF}
Let $p \in \PF{n}$ be a parking function such that its outcome $\OPF{n}{p}$ is a Hamiltonian path of the graph $G$. Then we have $p \in \FPF{G}$, and $\OFPF{G}{p} = \OPF{n}{p}$.
\end{proposition}

\begin{proof}
In fact, in this case the friendship and classical parking processes are identical. Indeed, when a car $i$ parks in some spot $k$ in the classical parking process, it will never move from spot $k$ thereafter. As such, since $\pi := \OPF{n}{p} \in \Ham{G}$, this implies that the spot $k$ is available at that time for car $i$ in the friendship parking process. Since by definition $k$ is the first value greater than or equal to $p_i$ which is unoccupied when $i$ enters the car park, car $i$ also parks in $k$ for the friendship parking process. This completes the proof.
\end{proof}

\begin{remark}\label{rem:PF_FPF_diff}
Propositions~\ref{pro:FPF_implies_PF} and \ref{pro:PF+outcome_implies_FPF} are not quite converse to each other. Indeed, there exist friendship parking functions (which are thus also classical parking functions) whose classical outcome is not a Hamiltonian path. Example~\ref{ex:FPF_valid} gives an example of such a parking function given by $p = (2,1,2,2)$. This is a friendship parking function for the $4$-cycle, but its classical outcome is $\OPF{4}{p} = 2134$, which is not a Hamiltonian path of the $4$-cycle (vertices $1$ and $3$ are not adjacent). As such, friendship parking functions are not just classical parking functions whose outcome is a Hamiltonian path.
\end{remark}

\subsection{Characterising friendship parking functions according to their outcomes}\label{subsec:general_counting}

In this section, we give a general characterisation and enumeration of friendship parking functions according to their outcome permutations. For $\pi \in S_n$, the \emph{outcome fibre} of $\pi$ is the set $\FibFPF{G}{\pi} := \{p \in \FPF{G}; \, \OFPF{G}{p} = \pi \}$ of $G$-parking functions whose (friendship) outcome is $\pi$. We begin by giving the definitions of \emph{blocker} and \emph{blocking sequence}, which will be used in the subsequent results.

\begin{definition}\label{def:blocking_sequence}
Let $\pi \in \Ham{G}$ be a Hamiltonian path of the graph $G$, and $i \in [n]$. We say that $j=\pi_k$ is a \emph{blocker} for $i$ in $\pi$, or equivalently that $i$ is blocked by $j$ in $\pi$, if either of the two following conditions holds:
\begin{enumerate}[topsep=2pt, itemsep=1pt]
  \item We have $j \leq i$; or
  \item We have $j > i$, but there is a neighbour $\ell$ of $j$ in $\pi$, i.e.\ $\ell \in \{\pi_{k-1}, \pi_{k+1} \}$ such that $\ell < i$ and $\ell$ is \emph{not} adjacent to $i$ in $G$.
\end{enumerate}
The \emph{blocking sequence} for $i$ in $\pi$ is the longest sub-sequence of $\pi$, ending at $i$, whose elements are all blockers for $i$. We denote $B(i, \pi, G)$ the blocking sequence for $i$, and $b(i, \pi, G) := \vert B(i, \pi, G) \vert $ its length.
\end{definition}

\begin{example}\label{exa:Hamil_path_G}
Let $G$ be the graph in Figure~\ref{fig:exa_Hamil_G}, with Hamiltonian path $\pi = 87152463$ (drawn in red). Then $i = 4$ is blocked by: $1, 2, 3, 4$ (all $\leq 4$),
as well as $5$ and $7$. For $5$ and $7$, note that they are both neighbours of $\ell = 1$ in $\pi$, which satisfies $\ell = 1 < 4 = i$, and $1$ is not adjacent to $4$ in $G$, so Condition~(2) of Definition~\ref{def:blocking_sequence} applies.

However, $j = 8$ is not a blocker, since its only neighbour $\ell = 7$ in $\pi$ is not strictly less than $i = 4$. Neither is $j = 6$ a blocker. Indeed, its neighbours in $\pi$ are $3$ and $4$. Among these, the only value which is strictly less than $i = 4$ is $\ell = 3$, and there is an edge between $\ell = 3$ and $i = 4$ in $G$, so Condition~(2) of Definition~\ref{def:blocking_sequence} does not apply here. The blocking sequence for $4$ is then $B(4, \pi, G) = 71524$, with length $b(4, \pi, G) = 5$.

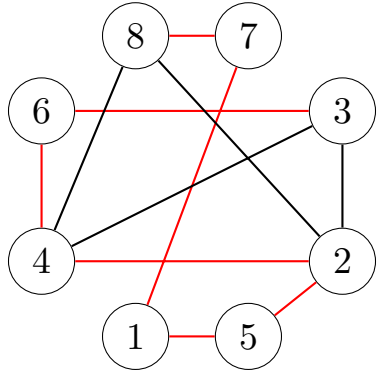
\begin{figure}[ht]
 \centering 
  \begin{tikzpicture}[scale=0.5]
  
    \node [draw, circle, scale=1.3pt] (1) at (-1.5,0) {$1$};
    \node [draw, circle, scale=1.3pt] (2) at (4,2) {$2$};
    \node [draw, circle, scale=1.3pt] (3) at (4, 6) {$3$};
    \node [draw, circle, scale=1.3pt] (4) at (-4, 2) {$4$};
    \node [draw, circle, scale=1.3pt] (5) at (1.5,0) {$5$};
    \node [draw, circle, scale=1.3pt] (6) at (-4, 6) {$6$};
    \node [draw, circle, scale=1.3pt] (7) at (1.5, 8) {$7$};
    \node [draw, circle, scale=1.3pt] (8) at (-1.5, 8) {$8$};
    \draw [thick, red] (8)--(7)--(1)--(5)--(2)--(4)--(6)--(3);
    \draw [thick] (8)--(4)--(3)--(2)--(8);
  \end{tikzpicture}
  \caption{The graph $G$ with Hamiltonian path $\pi=87152463$ (red)\label{fig:exa_Hamil_G}}

\end{figure}
\end{example}

We now state the main result of this section, which gives a complete characterisation of the outcome fibres of friendship parking functions in terms of blocking sequences.

\begin{theorem}\label{thm:charact_FPF_outcome}
Let $\pi \in \Ham{G}$ be a Hamiltonian path of $G$, and $p = (p_1, \cdots, p_n) \in [n]^n$ a parking preference. Then we have $p \in \FibFPF{G}{\pi}$ if, and only if, for all $i \in [n]$, we have $p_i \in S_i$, where $S_i := \{s \in [n]; \, \pi_s \in B(i, \pi, G) \}$ is the set of spots that are occupied in the outcome $\pi$ by an element in the blocking sequence for $i$ in $\pi$.
\end{theorem}

\begin{example}\label{ex:calc_FPF_outcome}
Consider the graph $G$ from Figure~\ref{fig:exa_Hamil_G} with corresponding Hamiltonian path $\pi=87152463$ as in Example~\ref{exa:Hamil_path_G}. We wish to calculate the fibre set $\FibFPF{G}{\pi}$ by applying Theorem~\ref{thm:charact_FPF_outcome}. For this we calculate the blocking sequences $B(i, \pi, G)$ for each element $i \in [8]$, and their corresponding index sets $S_i := \{s \in [n]; \, \pi_s \in B(i, \pi, G) \}$.
\begin{itemize}[itemsep=2pt]
\item For $i=1$, $B(1,\pi,G)= 1 = \pi_3$, so $S_1 = \{3\}$.
\item For $i=2$, $B(2,\pi,G)= 7152 = \pi_2 \pi_3 \pi_4 \pi_5$, so $S_2 = \{2, 3, 4, 5\}$.
\item For $i=3$, $B(3,\pi,G)= 3 = \pi_8$, so $S_3 = \{8\}$.
\item For $i=4$, $B(4,\pi,G)= 71524 = \pi_2 \cdots \pi_6$, so $S_4 = [2, 6]$.
\item For $i=5$, $B(5,\pi,G)= 15 = \pi_3 \pi_4$, so $S_5 = \{3, 4\}$.
\item For $i=6$, $B(6,\pi,G) = 715246 = \pi_2 \cdots \pi_7$, so $S_6 = [2, 7]$.
\item For $i=7$, $B(7,\pi,G) = 7 =\pi_2$, so $S_7 = \{ 2 \}$.
\item For $i=8$, $B(8,\pi,G) = 8 = \pi_1$, so $S_8 = \{ 1 \}$.
\end{itemize}

Finally, we get the fibre set $\FibFPF{G}{\pi} = \prod\limits_{i = 1}^8 S_i = \{3\} \times [2,5] \times \cdots \times \{2\} \times \{1\}$, and the fibre size $ \left\vert \FibFPF{G}{\pi} \right\vert = \prod\limits_{i = 1}^8\left\vert S_i \right\vert = 1*4*1*5*2*6*1*1 = 240.$

\end{example}

Theorem~\ref{thm:charact_FPF_outcome} immediately implies enumeration formulas for the fibres and total number of friendship parking functions. 

\begin{corollary}\label{cor:num_GFPF}
For any Hamiltonian path $\pi \in \Ham{G}$ of $G$, we have
$ \left\vert \FibFPF{G}{\pi} \right\vert = \prod\limits_{i\in [n]}b(i,\pi,G)$.
In particular, this implies that the total number of $G$-friendship parking functions is equal to
$\sum\limits_{\pi\in \Ham{G}}\left(\prod\limits_{i\in [n]}b(i,\pi,G)\right)$.
\end{corollary}

\begin{proof}[Proof of Theorem~\ref{thm:charact_FPF_outcome}]
Fix $\pi \in \Ham{G}$ a Hamiltonian path of $G$. For $i \in [n]$, denote $S_i := \{s \in [n]; \, \pi_s \in B(i, \pi, G) \}$ the set of spots occupied in $\pi$ by cars in the blocking sequence of $i$, and $s_i := \min S_i$ the first such spot. Note that by definition of the blocking sequence, we have $S_i = \left[s_i, \pi^{-1}_i \right]$.

We first assume that $p = (p_1, \cdots, p_n) \in \FibFPF{G}{\pi}$ is a friendship parking function whose outcome is the Hamiltonian path $\pi$. We wish to show that for all $i \in [n]$, we have $p_i \in S_i$. Seeking contradiction, suppose that is not the case. Let $i \in [n]$ be such that $p_i \notin S_i$, and define $j := \pi^{-1}_i$ to be the spot where car $i$ ends up. Since $S_i = \left[s_i, \pi^{-1}_i\right] = [s_i, j]$, we have two cases to consider.
\begin{enumerate}
  \item Case 1: $p_i > j$. That is, car $i$ prefers a spot (strictly) to the right of spot $j$. But by definition, we have $j = \pi_i$, so car $i$ should end up in spot $j$, which is impossible given the parking process (cars necessarily end up parking on or to the right of their preference).
  \item Case 2: $p_i < s_i$. Consider the spot $k := s_i - 1$, and the corresponding car $\pi_{k}$. Note that we have $p_i \leq k < j$. By definition of the blocking sequence, $\pi_k$ is not a blocker for $i$ in $\pi$. In particular, we must have $\pi_k > i$, so spot $k$ is unoccupied when car $i$ arrives. Moreover, none of the neighbouring spots $k-1, k+1$ can be occupied by cars arriving before $i$ which are not adjacent to $i$ in $G$. This means exactly that spot $k$ is available for $i$ in the $G$-friendship parking process for $p$. But since $p_i \leq k$, this means that car $i$ will end up either in spot $k$ or to its left. In particular, it is impossible for car $i$ to end up in spot $j$, since $j > k$. This gives the desired contradiction.
\end{enumerate}

We now show the converse. That is, suppose that $p = (p_1, \cdots, p_n)$ is a parking preference such that $p_i \in S_i$ for all $i \in [n]$. We wish to show that $p \in \FibFPF{G}{\pi}$. Consider the friendship parking process for $p$. We show by induction on $i \geq 1$ that car $i$ always ends up in spot $\pi^{-1}_i$.

For $i = 1$, we note that $1$ has no blockers other than itself, since it is the minimal element of $[n]$, and the condition for $j > i$ to be a blocker (Condition~(2) of Definition~\ref{def:blocking_sequence}) relies on the existence of an element $\ell$ satisfying $\ell < i$. That is, $B(1, \pi, G) = 1$, so $S_1 = \{ \pi^{-1}_1 \}$ by definition. Therefore, since $p_1 \in S_1$ by assumption, we have $p_1 = \pi^{-1}_1$. Finally, since car $1$ always parks in its preferred spot, we have that car $1$ ends up in spot $\pi^{-1}_1$, as desired.

Now suppose that, for some $i \geq 2$, we have shown that car $j$ parks in spot $\pi^{-1}_j$ for all $j \leq i-1$. Consider how car $i$ parks. By assumption, we have $p_i \in S_i = \left[ s_i, \pi^{-1}_i \right]$. Consider a spot $k \in S_i \setminus \{ \pi^{-1}_i \}$, and $j := \pi_k$ the corresponding car. We wish to show that $i$ cannot park in spot $k$. By definition, $j$ is a blocker for $i$, and $j \neq i$ (since $k = \pi^{-1}_j < \pi^{-1}_i$). We distinguish two cases.
\begin{enumerate}
  \item Case 1: $j < i$. By the induction hypothesis, $j$ occupies spot $\pi^{-1}_j = k$, so car $i$ cannot park in spot $k$.
  \item Case 2: $j > i$. In this case, by definition of blockers, for at least one of the values $\ell \in \{k-1, k+1\}$, we must have that $\pi_{\ell} < i$ and $\pi_{\ell}$ is not adjacent to $i$ in $G$. By the induction hypothesis, car $\pi_{\ell}$ occupies spot $\ell$, which neighbours spot $k$ in $\pi$, and since $\pi_{\ell}$ and $i$ are not adjacent in $G$, this prevents car $i$ from parking in spot $k$ in the friendship parking process.
\end{enumerate}
As such, car $i$ cannot park in any of the spots $s_i, s_i + 1, \cdots, \pi^{-1}_i - 1$. It therefore remains to show that it is allowed to park in spot $\pi^{-1}_i$ in the friendship parking process, i.e.\ that spot $\pi^{-1}_i$ is available for $i$. First, note that by the induction hypothesis, cars $1, \cdots i-1$ occupy spots $\pi^{-1}_1, \cdots, \pi^{-1}_{i-1}$. Therefore, when car $i$ enters the car park, spot $\pi^{-1}_i$ is unoccupied. Now consider $N(i) := \{ \pi^{-1}_i - 1, \pi^{-1}_i + 1 \}$ the two neighbouring spots to $\pi^{-1}_i$. By definition, if $s \in N(i)$, then $\pi_s$ is a neighbour of $i$ in the Hamiltonian path $\pi$, and therefore is adjacent to $i$ in the graph $G$. Moreover, by the induction hypothesis, when $i$ enters the car park, such a spot $s$ is either occupied by $\pi_s$ (if $\pi_s < i$), or unoccupied (if $\pi_s > i$). Finally, we have shown that when car $i$ enters the car park, spot $\pi^{-1}_i$ is unoccupied (see above), and each of its two neighbouring spots is either unoccupied, or is occupied by a car adjacent to $i$ in $G$. This means exactly that spot $\pi^{-1}_i$ is available for $i$ in the friendship parking process, as desired. 
\end{proof}

\subsection{Application: the cycle graph case}\label{subsec:cycle}

In this section, we are interested in the specific case where $G$ is a cycle graph. For $n \geq 3$, we denote $C_n$ the cycle graph on $n$ vertices. This is the graph with vertex set $[n]$ and edge set $\{ (i,i+1) \}_{i \in [n-1]} \cup \{ (n,1) \}$ (see Figure \ref{fig:cycle}).

\begin{figure}[ht]
 \centering 
  \begin{tikzpicture}[scale=0.65]
  
    \node [draw, circle] (1) at (-1.35,0) {$1$};
    \node [draw, circle] (2) at (1.35,0) {$2$};
    \node [draw, circle] (3) at (2.5,2) {$3$};
    \node [draw, circle] (4) at (1.35,4) {$4$};
    \node [draw, circle] (5) at (-1.35,4) {$5$};
    \node [draw, circle] (6) at (-2.5,2) {$6$};
    \draw [thick] (1)--(2)--(3)--(4)--(5)--(6)--(1);
  \end{tikzpicture}
  \caption{The cycle graph $C_6$ on six vertices.\label{fig:cycle}}

\end{figure}
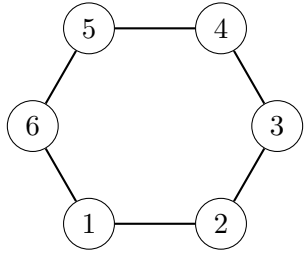

We wish to apply Theorem~\ref{thm:charact_FPF_outcome} and Corollary~\ref{cor:num_GFPF} to the cycle graph case. First, note that Hamiltonian paths of the cycle graph are of one of the following two types:
  \begin{enumerate}
    \item \emph{Increasing cycles}, which are of the form $\IncCyc{i}{n} := \inccyc{i}{n}$ for some $i \in [n]$;
    \item \emph{Decreasing cycles}, which are of the form $\DecCyc{i}{n} := \deccyc{i}{n}$ for some $i \in [n]$.
  \end{enumerate}
We characterise the blocking sequences for both types of outcome. We begin with the decreasing cycle case. For technical reasons (see Remark~\ref{rem:n_geq_4}) we assume that $n \geq 4$.

\begin{proposition}{Decreasing cycle case, $n \geq 4$.}\label{pro:dec_cycle_blocking}
\begin{enumerate}
  \item \textbf{Case $i \leq n-2$}. Let $i \leq n-2$. For any $j \in [n]$, we have:
  $$ B\left( j, \DecCyc{i}{n}, C_n \right) = 
  \left\{
  \begin{array}{ll}
      j & \text{ if } j \leq n-2, \\
      i(i-1) \cdots 1n(n-1) & \text{ if } j = n-1, \\
      i(i-1) \cdots 1n & \text{ if } j = n.
  \end{array}
  \right.$$
  In particular, this implies that $\left\vert \FibFPF{C_n}{\DecCyc{i}{n}} \right\vert = (i+1)(i+2)$.
  \item \textbf{Case $i = n-1$}. For any $j \in [n-1]$, we have $B\left( j, \DecCyc{n-1}{n}, C_n \right) = j$. Moreover, we have $B \left( n, \DecCyc{n-1}{n}, C_n \right) = (n-1) \cdots 1n = \DecCyc{n-1}{n}$. In particular, this implies that $\left\vert \FibFPF{C_n}{\DecCyc{n-1}{n}} \right\vert = n$.
  \item \textbf{Case $i = n$}. For any $j \in [n]$, we have $B \left( j, \DecCyc{n}{n}, C_n \right) = j$. In particular, this implies that $\left\vert \FibFPF{C_n}{\DecCyc{n}{n}} \right\vert = 1$.
\end{enumerate}
\end{proposition}

\begin{proof}
Fix some $i, j \in [n]$, and consider the blocking sequence for $j$ in the decreasing cycle $\DecCyc{i}{n} = \deccyc{i}{n}$. We first note that if $j \leq n-2$, then $\DecCyc{i}{n} = \cdots (j+2)(j+1)j \cdots$ (assuming $j$ doesn't appear in position $1$ or $2$). This implies that $j+1$ is not a blocker for $j$, since it is strictly greater than $j$ and both of its neighbours in $\DecCyc{i}{n}$ are greater than or equal to $j$. This shows that $B\left(j, \DecCyc{i}{n}, C_n \right) = j$, as desired.

Now consider $j = n-1$. If $i \leq n-2$, then we have $\DecCyc{i}{n} = i \cdots 1n(n-1) \cdots$. As the vertices $1$ and $n-1$ are not adjacent in $C_n$ (since $n \geq 4$), and $1 < n-1$, this implies that $n$ is a blocker for $n-1$. Moreover, all values $i, i-1, \cdots, 1$ are strictly less than $n-1$, so are all blockers. Thus, we have $B\left(n-1, \DecCyc{i}{n}, C_n \right) = i(i-1) \cdots 1n(n-1)$, as desired. However, this relies on $1$ (immediately) preceding $n$ in $\DecCyc{i}{n}$, which only occurs if $i \leq n-2$. Otherwise, $n-1$ appears in first or second position, and is at most preceded by $n$, which is not a blocker in that case, so $B\left(n-1, \DecCyc{i}{n}, C_n \right) = n-1$ if $i \in \{n-1, n\}$, as desired.

Finally, consider $j = n$. If $i \leq n-1$, then we have $\DecCyc{i}{n} = i \cdots 1n \cdots$. All elements in $i, i-1, \cdots, 1$ are strictly less than $n$, so are blockers in this case. This gives $B\left(n, \DecCyc{i}{n}, C_n \right) = i(i-1) \cdots 1n $ as desired. If $i = n$, then $n$ appears in first position in $\DecCyc{n}{n} = n(n-1) \cdots 1$, so $B\left(n, \DecCyc{i}{n}, C_n \right) = n$. This completes the proof.
\end{proof}

\begin{remark}\label{rem:n_geq_4}
The above proof in the case $j = n-1$ shows why we assumed that $n \geq 4$. Otherwise, for $n=3$, in the outcome $132$, the blocking sequence for $j = n-1 = 2$ is just $2$. Indeed, in this case $3$ is not a blocker for $2$, since $1$ and $2$ are adjacent in $C_3$. We therefore get $\left\vert \FibFPF{C_3}{132} \right\vert = 1*1*b(3, 132, C_3) = 2$, instead of the value $2*3 = 6$ as would be given by Case~(1) of Proposition~\ref{pro:dec_cycle_blocking}.
\end{remark}

\begin{proposition}{Increasing cycle case.}\label{pro:inc_cycle_blocking}
\begin{enumerate}
  \item \textbf{Case $i \geq 4$}. Let $i \geq 4$. For any $j \in [n]$, we have:
  $$
  B \left( j, \IncCyc{i}{n}, C_n \right) = \left\{
    \begin{array}{ll}
      i(i+1) \cdots j & \text{ if } j \geq i, \\
      n12 \cdots j & \text{ if } 3 \leq j \leq i-1, \\
      12 & \text{ if } j=2, \\
      1 & \text{ if } j=1.
    \end{array}
  \right.
  $$
  In particular, this implies that $\left\vert \FibFPF{C_n}{\IncCyc{i}{n}} \right\vert = (n-i+1)!i!/3$.
  \item \textbf{Case $i \leq 3$}. Let $i \leq 3$. For any $j \in [n]$, we have:
  $$
  B \left( j, \IncCyc{i}{n}, C_n \right) = \left\{
    \begin{array}{ll}
      i \cdots j & \text{ if } j \geq i, \\
      1 \cdots j & \text{ if } j \leq i-1.
    \end{array}
  \right.
  $$
  In particular, this implies that $\left\vert \FibFPF{C_n}{\IncCyc{i}{n}} \right\vert = (n+1-i)!(i-1)!$.
\end{enumerate}
\end{proposition}

\begin{proof}
Fix $i,j \in [n]$. If $j \geq i$, we have $B \left( j, \IncCyc{i}{n}, C_n \right) = i(i+1)\cdots j$. Indeed, all elements to the left of $j$ in $\IncCyc{i}{n}$ are less than or equal to $j$, so they are all blockers for $j$. Moreover, it is straightforward to check that $B \left( 1, \IncCyc{i}{n}, C_n \right) = 1$ (true in fact for any graph $G$ and Hamiltonian path $\pi \in \Ham{G}$, see proof of Theorem~\ref{thm:charact_FPF_outcome}).

Now consider the case $j=2$ where $j \leq i-1$, i.e.\ $2$ is not in first position in $\IncCyc{i}{n}$. In this case, $2$ immediately follows $1$ in $\IncCyc{i}{n}$, and therefore $1$ is in the blocking sequence for $2$. Moreover, $n$ is not a blocker for $2$, since its neighbours in $\IncCyc{i}{n}$ are $n-1$ (which is greater than or equal to $2$), and $1$ (which is adjacent to $2$ in $C_n$). Therefore if $i \geq 3$ we have $B \left( 2, \IncCyc{i}{n}, C_n \right) = 12 = 1 \cdots j$, as desired.

It therefore remains to consider the case $3 \leq j \leq i-1$. Necessarily we have $i \geq 4$, so that $\IncCyc{i}{n} = \cdots n12 \cdots j \cdots$. By definition, $1, 2, \cdots, j$ are all blockers for $j$, since they are less than or equal to $j$. Moreover, since $3 \leq j \leq i-1 \leq n-1$, the elements $j$ and $1$ are not adjacent in $C_n$ ($n$ and $2$ are the vertices adjacent to $1$). Since $n$ is followed by $1$ in $\IncCyc{i}{n}$, this implies that $n$ is a blocker for $j$. Finally, $n-1$ is not a blocker for $j$, since both its neighbours in $\IncCyc{i}{n}$ are greater than or equal to $j$. Therefore, we get $B \left( 2, \IncCyc{i}{n}, C_n \right) = n12 \cdots j$, as desired. This completes the proof.
\end{proof}

Combining Propositions~\ref{pro:dec_cycle_blocking} and \ref{pro:inc_cycle_blocking}, we get the following enumeration for the total number of $C_n$-friendship parking functions.

\begin{theorem}\label{thm:count_cn_PFP}
Let $n \geq 4$. The total number of friendship parking functions for the cycle graph $C_n$ is given by
$$ \big\vert \FPF{C_n} \big\vert = n + 1 + \sum\limits_{i=1}^{n-2} (i+1)(i+2) + \sum\limits_{i=1}^{3} (n+1-i)!(i-1)! + \sum\limits_{i=4}^n (n-i+1)!i!/3.$$
\end{theorem}

For the case $n=3$, the enumeration is given by the same formula, except that we replace the product $(i+1)(i+2)$ in the left-most sum with just $(i+1)$, as per Remark~\ref{rem:n_geq_4}.

\section{Cyclic parking functions and permutation components}\label{sec:cyclicPF}

In this part we consider classical parking functions whose outcome is an increasing cycle, i.e.\ of the form $\IncCyc{i}{n} := \inccyc{i}{n}$ for some $i \in [n]$. This is a slight variant on the $C_n$-friendship parking functions from Section~\ref{subsec:cycle}. We say that $p = (p_1, \cdots, p_n)$ is a \emph{cyclic parking function} if we have $\OPF{n}{p} = \IncCyc{i}{n}$, for some $i \in [n]$. We denote $\CycPF{n} := \bigcup\limits_{i \in [n]} \FibPF{n}{\IncCyc{i}{n}}$ the set of cyclic parking functions of length $n$. 

\begin{proposition}\label{pro:outcome_enum_CycPF}
For any $i \in [n]$, we have $\left\vert \FibPF{n}{\IncCyc{i}{n}} \right\vert = (n+1-i)! (i-1)!$.
\end{proposition}

\begin{proof}
The proof is similar to that of Theorem~\ref{thm:charact_FPF_outcome} and Proposition~\ref{pro:inc_cycle_blocking}, so we allow ourselves to be briefer here. Let $i \in [n]$, and consider how a parking preference $p = (p_1, \cdots, p_n)$ can yield the outcome permutation $\IncCyc{i}{n} = \inccyc{i}{n}$ (for the classical parking process). Car $1$ ends up in spot $j := n+2-i$ (with the convention that $j = 1$ if $i=1$). Since car $1$ always ends up in its preferred spot in the classical parking process, we must have $p_1 = j$. Then car $2$ ends up in spot $j+1$. This occurs if, and only if, we have $p_2 \in \{j, j+1\}$ (since car $1$ occupies $j$). 

We repeat this reasoning on the first $i-1$ cars, and see that each car $k \in \{1, \cdots, i-1\}$ has exactly $k$ possibilities for its parking preference. This gives the $(i-1)!$ factor in the formula. The $(n+1-i)!$ factor comes from an analogous reasoning on cars $i, i+1, \cdots, n$. Car $i$ ends up in spot $1$, so must prefer that spot. Car $(i+1)$ ends up in spot $2$, so must prefer spot $1$ (occupied by car $i$) or spot $2$, and so on. This completes the proof.
\end{proof}

\begin{remark}\label{rem:CycPF_not_FPF(C_n)}
The formula in Proposition~\ref{pro:outcome_enum_CycPF} is quite similar to that of friendship parking functions on the cycle graph $C_n$ whose outcome is an increasing cycle (see Proposition~\ref{pro:inc_cycle_blocking}). The key difference lies in the role of the car/vertex $n$. In $C_n$-friendship parking functions, this vertex plays the role of a blocker for cars $j \geq 3$ which park to its right, whereas in classical parking functions it is never a ``blocker'' (if car $j$ wanted to park in the corresponding spot, it could). 

As an example, consider the parking function $p = (2,3,1,1)$. In the friendship parking process for $C_4$, car $3$ cannot park in spot $1$, since spot $2$ is occupied by car $1$, and there is no edge $(1,3)$ in $C_4$. Therefore we get the outcome $\OFPF{4}{p} = 4123 = \IncCyc{4}{4}$, and $p$ is a $C_4$-friendship parking function. However, in the classical parking process, car $3$ can park in spot $1$, and so car $4$ will end up in spot $4$, yielding the outcome $\OPF{4}{p} = 3124$, so $p$ is not a cyclic parking function in the sense of this section.
\end{remark}

Summing the right-hand side of the enumeration in Proposition~\ref{pro:outcome_enum_CycPF} over values of $i$ from $1$ to $n$, combined with the change of variable $i \rightarrow n+1-i$, immediately yields the following.

\begin{corollary}\label{cor:total_number_CycPF}
We have $\left\vert \CycPF{n} \right\vert = \sum\limits_{i=0}^{n-1} i! (n-i)!$.
\end{corollary}

Calculating these numbers for the first few values of $n$, we get the sequence $1, 3, 10, 40, 192, \cdots$, which corresponds to Sequence~A136128 in the OEIS~\cite{OEIS}. The entry for this sequence states ``Number of components in all permutations of $[n]$'', and the goal of this section is to give a bijective proof of this fact by exhibiting a bijection between cyclic parking functions and permutation components. As such, cyclic parking functions provide a new combinatorial interpretation for this Sequence~A136128. Before defining our bijection, we first need to introduce some terminology relating to permutations.

Given a permutation $\pi = \pi_1 \cdots \pi_n$, a \emph{component} of $\pi$ is a sub-sequence $\pi_i \pi_{i+1} \cdots \pi_j$ for some $1 \leq i \leq j \leq n$ such that $\{ \pi_i, \pi_{i+1}, \cdots \pi_j \} = \{ i, i+1, \cdots, j\}$, which is minimal for inclusion (i.e.\ it does not contain any smaller such sub-sequence). It is straightforward to see that every permutation can be uniquely decomposed into its components. For example, the permutation $31248657$ has three components: $312$, $4$, $8657$. We will use a slash to separate the components of a permutation, for example $312/4/8657$. The set of components of a permutation $\pi$ is denoted $\comp{\pi}$.

Note that the same sub-sequence can be a component of more than one permutation. For example, $1$ is a component of both $1/32$ or $1/2/3$. In this section, we consider the set of all components of permutations of length $n$, $\Comp{n} := \bigsqcup\limits_{\pi \in S_n} \comp{\pi}$, where $\bigsqcup$ denotes the disjoint union. That is, we distinguish between the component $1$ in the permutation $1/32$ and the component $1$ in the permutation $1/2/3$. For $c \in \Comp{n}$, we will denote $\Perm{c} \in S_n$ the underlying permutation, i.e.\ $c \in \comp{\Perm{c}}$).

In our bijection between cyclic parking functions and permutation components, permutation inversions will play a key role. Given a permutation $\pi = \pi_1 \cdots \pi_n \in S_n$, we say that an ordered pair $(i, j)$ is an \emph{inversion} of $\pi$ if $i > j$ and $i$ appears before $j$ in the permutation $\pi$. The \emph{inversion number} of $i$ in $\pi$, denoted $\inv{i}{\pi}$, is the number of $j < i$ such that $(i, j)$ is an inversion of $\pi$. In words, the inversion number of $i$ counts how many letters less than $i$ appear after $i$ in $\pi$. For example, the permutation $\pi = 3152476$ has the following inversions: $(3,1), (3,2), (5,2), (5,4), (7,6)$. We have $\inv{1}{\pi} = \inv{2}{\pi} = \inv{4}{\pi} = \inv{6}{\pi} = 0$, $\inv{3}{\pi} = \inv{5}{\pi} = 2$, and $\inv{7}{\pi} = 1$. Note that any inversion of $\pi$ is necessarily contained in a component of $\pi$.

\begin{definition}\label{def:inversion sequences}
Let $a = (a_1, \cdots, a_n)$ be a sequence of non-negative integers. We say that $a$ is an \emph{inversion sequence} if $a_i < i$ for all $i \in [n]$.
\end{definition}

The terminology \emph{inversion sequence} comes from the following classical result, of which we give a brief proof for completeness.

\begin{theorem}\label{thm:bij_perm_invseq}
The map $\mathrm{InvSeq} : \pi = \pi_1 \cdots \pi_n \mapsto \InvSeq{\pi} := \left( \inv{1}{\pi}, \cdots, \inv{n}{\pi} \right)$ is a bijection from the set of $n$-permutations to the set of inversion sequences of length $n$.
\end{theorem}

\begin{proof}
Let $\pi = \pi_1 \cdots \pi_n \in S_n$, and $\InvSeq{\pi} := \left( \inv{1}{\pi}, \cdots, \inv{n}{\pi} \right)$. By definition, for any $i \in [n]$, $\inv{i}{\pi}$ is the number of $j < i$ such that $(i,j)$ is an inversion of $\pi$, which immediately implies that $\inv{i}{\pi} < i$, i.e.\ that $\InvSeq{\pi}$ is an inversion sequence. Moreover, it is straightforward to see that there are $n!$ inversion sequences of length $n$. Indeed, we have $1$ choice ($0$) for the entry $a_1$, $2$ choices ($0$ and $1$) for the entry $a_2$, and so on.

To show that the map $\mathrm{InvSeq}$ is a bijection, it is therefore sufficient to show that it is surjective. Let $a = (a_1, \cdots, a_n)$ be an inversion sequence. We construct a permutation $\pi$ as follows. Start with $i=0$ and the empty word $\pi^0 = \emptyset$. Suppose that after $i$ steps we have constructed a word $\pi^i = \pi^i_1 \cdots \pi^i_i \in S_i$ such that $\inv{j}{\pi^i} = a_j$ for all $j \in [i]$. Then at step $i+1$, we define $\pi^{i+1} = \pi^{i+1}_1 \cdots \pi^{i+1}_{i+1}$ by simply inserting $(i+1)$ in the permutation $\pi^i$ at index $i+1-a_{i+1}$ (i.e.\ $a_{i+1}$ spots starting from the right). By construction, we get a permutation $\pi^{i+1} \in S_{i+1}$ such that $\inv{j}{\pi^{i+1}} = a_j$ for all $j \in [i+1]$. Therefore, the final permutation $\pi^n \in S_n$ obtained satisfies $\InvSeq{\pi^n} = a$, as desired.
\end{proof}

\begin{example}\label{ex:inv_seq_to_perm}
Consider the inversion sequence $a = (0, 0, 1, 3, 2)$. We start with the empty word $\pi^0 = \emptyset$. At step $1$, we insert the letter $1$ in the empty word so that there are $a_1 = 0$ letters to its right, i.e.\ we get $\pi^1 = 1$ (we will always have this step). At step $2$, we have $a_2 = 0$, so we insert $2$ into the permutation $\pi^1 = 1$, so that there are $0$ letters to its right, yielding $\pi^2 = 12$. Then at step $3$, we have $a_3 = 1$. As such, we insert $3$ into the permutation $12$ so that it has $1$ letter to its right, yielding $\pi^3 = 132$.
At step $4$, we insert $4$ into $132$ so that it has $a_4 = 3$ letters to its right, yielding $\pi^4 = 4132$. Finally, at step $5$, we insert $5$ so that it has $a_5=2$ letters to its right, yielding the permutation $\pi = \pi^5 = 41532$. We can then check that $\InvSeq{\pi} = (0, 0, 1, 3, 2) = a$, as desired.
\end{example}

Inversion sequences are linked to the displacement vector of parking functions by the following result.

\begin{proposition}\label{pro:disp_inv_seq}
Let $p \in \PF{n}$ be a parking function. Then its displacement vector $\disp{p}$ is an inversion sequence.
\end{proposition}

\begin{proof}
Let $p \in \PF{n}$, and write $\disp{p} = (d_1, \cdots, d_n)$ for its displacement vector. Fix $i \in [n]$. By definition, $d_i$ is the distance car $i$ travels from its preferred spot to its final parking spot. In the classical parking process, this means that spots $p_i, p_i+1, \cdots, p_i + d_i - 1$ must all be occupied when car $i$ enters the car park. Since there are $i-1$ cars entering before car $i$, and each of these occupies at most one of those spots, this implies that $d_i \leq i-1 < i$, as desired.
\end{proof}

We are now equipped to define the bijection from cyclic parking functions to permutation components. Given a cyclic parking function $p$, whose outcome is $\IncCyc{i}{n}$ for some $i \in [n]$, we define $\Psi(p)$ to be the component of $\pi$ containing $i$, where $\pi \in S_n$ is the unique permutation such that $\InvSeq{\pi} = \disp{p}$.

\begin{example}\label{ex:one_bij_cycPF_perm_comp}
Consider the parking function $p = (4, 4, 6, 6, 7, 9, 7, 1, 2, 1)$. By running the parking process, we get $\OPF{10}{p} = 89(10)1234567 = \IncCyc{8}{10}$. The displacement vector of $p$ is $\disp{p} = (0, 1, 0, 1, 1, 0, 3, 0, 0, 2)$. The corresponding permutation $\pi$ such that $\disp{p} = \InvSeq{\pi}$ is $\pi = 21/47536/(10)89$, and we map this to the component containing car $i = 8$, so $\Psi(p) = (10)89$.
\end{example}

\begin{example}\label{ex:all_bij_cycPF_perm_comp}
Here we fix $n=3$, and consider the images of all cyclic parking functions of length $n$ under the map $\Psi$. These are illustrated in Table~\ref{tab:all_bij_cycPF_perm_comp} below. We have partitioned cyclic parking functions (listed in Column~2) according to their outcome (in Column~1). Column~3 then gives the displacement vector of each parking function, with Column~4 listing the permutation whose inversion sequence is the displacement vector. Finally, in the component column (Column~5), we have copied in this permutation for clarity, with the corresponding component $\Psi(p)$ in \textbf{bold}.
\begin{table}[ht]
    \centering
    \begin{tabular}{c|c|c|c|c}
         Outcome & Parking function $p$ & $\disp{p}$ & Permutation $\pi$ & Component $c = \Psi(p)$ \\
         \hline
         \multirow{6}{2em}{$123$} & $(1, 1, 1)$ & $(0, 1, 2)$ & $321$ & $\mathbf{321}$\\
          & $(1, 1, 2)$ & $(0, 1, 1)$ & $231$ & $\mathbf{231}$\\
          & $(1, 1, 3)$ & $(0, 1, 0)$ & $213$ & $\mathbf{21}/3$\\
          & $(1, 2, 1)$ & $(0, 0, 2)$ & $312$ & $\mathbf{312}$\\
          & $(1, 2, 2)$ & $(0, 0, 1)$ & $132$ & $\mathbf{1}/32$\\
          & $(1, 2, 3)$& $(0, 0, 0)$ & $123$ & $\mathbf{1}/2/3$ \\
          \hline
         \multirow{2}{2em}{$231$} & $(3, 1, 1)$ & $(0, 0, 1)$ & $132$ & $1/\mathbf{32}$\\
          & $(3, 1, 2)$ & $(0, 0, 0)$ & $123$ & $1/\mathbf{2}/3$\\
          \hline
         \multirow{2}{2em}{$312$} & $(2, 2, 1)$ & $(0, 1, 0)$ & $213$ & $21/\mathbf{3}$\\
          & $(2, 3, 1)$ & $(0, 0, 0)$ & $123$ & $1/2/\mathbf{3}$\\
    \end{tabular}
    \caption{All cyclic parking functions of length $n=3$, with their corresponding permutation components.}
    \label{tab:all_bij_cycPF_perm_comp}
\end{table}
\end{example}

\begin{theorem}\label{thm:bij_cycPF_perm_comp}
The map $\Psi : \CycPF{n} \rightarrow \Comp{n}$ is a bijection from cyclic parking functions to permutation components. Moreover, for any cyclic parking function $p$ and any $i \in [n]$, the displacement of car $i$ in $p$ is equal to the inversion number of $i$ in the underlying permutation $\Perm{\Psi(p)}$.
\end{theorem}

\begin{lemma}\label{lem:first_car_min_comp}
Let $p \in \CycPF{n}$ be a cyclic parking function with outcome $\OFPF{C_n}{p} = \IncCyc{i}{n}$, i.e.\ car $i$ ends up in spot $1$. Then we have $i = \min \Psi(p)$.
\end{lemma}

\begin{proof}
Let $p \in \CycPF{n}$ be a cyclic parking function, and $\disp{p} = (d_1, \cdots, d_n)$ its displacement vector. Define $\pi \in S_n$ to be the unique permutation such that $\InvSeq{\pi} = \disp{p}$. Denote $c := \Psi(p)$ the component of the permutation $\pi$ corresponding to the parking function $p$, and let $i' := \min c$ and $k = \vert c \vert$ be its minimal element and size respectively. By definition, the component $c$ is a permutation of the set $\{ i', \cdots, i'+k-1\}$. Moreover, by construction we have $d_j = \inv{j}{\pi} = \inv{j}{c}$ for any $j \in c$, since all inversions of the permutation $\pi$ are necessarily contained in its components.

Seeking contradiction, suppose that $i' < i$. We consider cars $i', i'+1, \cdots, i-1$ in the parking process, which are the $i-i'$ smallest elements of the component $c$. By the remarks above, we have $d_{i'+j-1} < j$ for all $j \in [i-i']$. Now consider cars $i, i+1, \cdots, i'+k-1$. By definition of $i$, these occupy spots $1, \cdots, k+i'-i$. In particular, we must have $\disp{i+j-1} < j$ for all $j \in [k+i'-i]$. These inequalities summarise as below.
$$
\begin{array}{rcccccc}
     \text{cars:} & i' & \cdots & i-1 & i & \cdots & i'+k-1 \\
     \hline
      & d_{i'} & \cdots & d_{i-1} & d_i & \cdots & d_{i'+k-1} \\
      \text{displacement:} & \wedge & \cdots & \wedge & \wedge & \cdots & \wedge \\
       & 1 & \cdots & i-i' & 1 & \cdots & k + i' - i \\
\end{array}
$$
Now recall that $\disp{p} = \InvSeq{\pi}$, and that $c$ is a sub-sequence in $\pi$ which is a permutation of the set $\{i' \cdots, i' + k-1\}$. It is straightforward, using the inverse construction from Theorem~\ref{thm:bij_perm_invseq}, to check that the inequalities above imply that $c$ can be written as $c = c^1 \! \cdot \! c^2$, where $c^1$, resp.\ $c^2$, is a permutation of the set $\{i', \cdots, i-1\}$, resp.\ of the set $\{i, \cdots, i'+k-1\}$. This contradicts the minimality of the component $c$.
\end{proof}

\begin{proof}[Proof of Theorem~\ref{thm:bij_cycPF_perm_comp}]
We first show that $\Psi$ is injective. For this, let $p, p'$ be two cyclic parking functions such that $c := \Psi(p) = \Psi(p')$. Let $i := \min c$. By Lemma~\ref{lem:first_car_min_comp}, we have that $\OFPF{n}{p} = \OFPF{n}{p'} = \IncCyc{i}{n}$. Moreover, by construction, we have $\disp{p} = \disp{p'} = \Perm{c}$. We claim that two parking functions which share the same outcome and displacement vector are necessarily equal. Indeed, the preference of any car is equal to its index in the outcome permutation minus its displacement (this is the definition of the displacement vector).

It remains to show that $\Psi$ is surjective, which we do by constructing its inverse. Let $c \in \comp{\pi}$ be a component of some permutation $\pi \in S_n$. Define $i := \min c$ to be its minimal element. Let $s = (s_1, \cdots, s_n) := \InvSeq{\pi}$ be the inversion sequence of $\pi$. For $j \in [n]$, denote $\overline{j}$ the index of $j$ in the cyclic permutation $\IncCyc{i}{n}$. That is, $\bar{j} = n+j+1-i$ if $j \leq i-1$, and $\bar{j} = j+1-i$ if $j \geq i$.

We claim that $s_j < \bar{j}$ for all $j \in [n]$. If $j \leq i-1$, this simply follows from the fact that $s$ is an inversion sequence: we have $\bar{j} \geq j +1$ and $s_j < j$. Now note that since $i$ is the minimal element in a component of the permutation $\pi$, this means in particular that $\pi$ can be decomposed as $\pi^{<i} \! \cdot \! \pi^{\geq i}$, where $\pi^{<i}$, resp.\ $\pi^{\geq i}$, is a permutation of the set $\{1, \cdots, i-1\}$, resp. $\{i, \cdots, n\}$. In particular, if $j \geq i$, all inversions $(j,j')$ of $\pi$ must be in the right-hand permutation $\pi^{ \geq i }$. This implies in particular that $s_j < j-(i-1) = \bar{j}$, as desired.

We can then define $p = (p_1, \cdots, p_n)$ by setting $p_j = \bar{j} - s_j$ for all $j \in [n]$. By the above claim, we have $p_j > 0 $ for all $j$, so $p$ is a parking preference. Moreover, by construction, we have $\OFPF{n}{p} = \IncCyc{i}{n}$, and $\Psi(p) = c$. This shows that $\Psi$ is surjective, and thus a bijection, as desired. The fact that $\Psi$ maps the displacement of a car to its inversion number in the underlying permutation is by construction: the permutation $\Perm{\Psi(p)}$ is defined so that $\InvSeq{\Perm{\Psi(p)}} = \disp{p}$.
\end{proof}

\begin{example}\label{ex:psi^-1}
We show an example of the construction of $\Psi^{-1}$ in the previous proof. Consider the permutation $\pi = 3412/\mathbf{7865}/9$, with component $c = 7865$. We first calculate the inversion sequence of $\pi$, yielding $\InvSeq{\pi} = (0, 0, 2, 2; 0, 1, 2, 2; 0)$ where we have used a semi-colon to separate the permutation components for clarity. The minimal element of $c$ is $i = 5$. This indicates that car $5$ must end up in spot $1$, cars $6$ to $9$ end up in spots $2$ to $5$, and cars $1$ to $4$ end up in spots $6$ to $9$. For each car $i$, we determine its parking preference $p_i$ by subtracting its inversion number (which will be equal to its displacement) from its final spot. For example, car $8$ ends up in spot $4$, and $8$ has inversion number $2$ in $\pi$, corresponding to the inversions $(8,6)$ and $(8,5)$, so we set $p_8 = 4 - 2 = 2$. We do this for each car, and finally get the parking preference $p = (6,7,6,7,1,1,1,2,5)$. We can check that $\OFPF{9}{p} = 567891234$, and that $\disp{p} = \InvSeq{\pi}$, as desired.
\end{example}

\begin{remark}\label{rem:inv_seq_disp}
Theorem~\ref{thm:bij_cycPF_perm_comp} implies a converse for Proposition~\ref{pro:disp_inv_seq}. Namely, if $a = (a_1, \cdots, a_n)$ is an inversion sequence, then there exists a cyclic parking function $p$ such that $\disp{p} = a$. In fact, the number of such cyclic parking functions is equal to the number of components of the permutation $\pi$ such that $a = \InvSeq{\pi}$. Indeed, given an inversion sequence $a$ with corresponding permutation $\pi$ (i.e.\ $\InvSeq{\pi} = a$), then the set of cyclic parking functions $p$ satisfying $\disp{p} = a$ is exactly the set $\Psi^{-1} \left( \comp{\pi} \right)$.
\end{remark}

\section{Conclusion and future work}\label{sec:future}

In this paper, we introduced a new variation on classical parking functions, called \emph{friendship parking functions}. 
These essentially follow classical parking rules, with the added restriction that a car can only park next to its ``friends'' (adjacent vertices in an underlying graph $G$). 
We gave a full characterisation of $G$-friendship parking functions according to their outcome, or equivalently Hamiltonian paths of $G$ (Theorem~\ref{thm:charact_FPF_outcome}). 
We then applied this result and its enumerative consequence to the case where $G$ is the cycle graph on $n$ vertices (Section~\ref{subsec:cycle}).
Finally, in Section~\ref{sec:cyclicPF} we returned to classical parking functions, considering \emph{cyclic parking functions} whose (classical) outcome is an increasing cycle of the form $\inccyc{i}{n}$. We exhibited (Theorem~\ref{thm:bij_cycPF_perm_comp}) a bijection between cyclic parking functions of \emph{permutation components}, which maps the \emph{displacement} of each car to its \emph{inversion} number in the corresponding permutation.

One natural question from this last work is to ask whether other families of permutations have natural combinatorial interpretations of their outcome fibres (whether for classical parking functions or one of the variations thereon). For example, in~\cite{SeligZhuMVP} it was shown that the outcome fibre of a permutation $\pi$ can be interpreted in terms of certain subgraphs of the corresponding permutation graph $G_{\pi}$, and various families of permutations were studied under that representation. It would be interesting to ask similar questions of classical parking functions.

Regarding friendship parking functions, a first question to consider would be to pursue the study of graph families. In this paper we studied the case of cycle graphs. We could also investigate other families such as cycle graphs with a different labelling, wheel graphs, complete graphs with one edge removed, and so on. We should look for both closed enumeration formulas and combinatorial interpretations of friendship parking functions in those cases.

Finally, we could also consider the case where the underlying graph $G$ does not have a Hamiltonian path. In this case, there will be no friendship parking functions allowing all cars to park, but we could consider \emph{defective} parking functions, as in the classical case~\cite{CamDefPF}. In other words, for some fixed $k > 0$, how many friendship parking functions are there where exactly $k$ cars fail to park? What are the minimal/maximal values of such a $k$? Of particular interest could be the case where $G$ is a disjoint union of cliques, i.e.\ where the friendship relationship is transitive. In other words, we partition the set $[n]$ into subsets $F_1, \cdots, F_j$ where any pair of elements in $F_{\ell}$ are friends, but there is no friendship between elements of distinct $F_{\ell}$ and $F_{\ell'}$. What can be said about friendship parking functions under this model?


\section*{Acknowledgments}

The research leading to these results was partially supported by the National Natural Science Foundation of China (NSFC) under Grant Agreement No 12101505, and by the Summer Undergraduate Research Fellowship programme at Xi'an Jiaotong-Liverpool University. The authors would like to thank Xiaoyi Zhang for helpful initial discussions. The authors have no competing interests to declare that are relevant to the content of this article.


\bibliographystyle{plain}
\bibliography{FPF_bib}

\end{document}